\documentclass[psamsfonts]{amsart}

\usepackage{amssymb,amsfonts}
\usepackage[all,arc]{xy}
\usepackage{enumerate}
\usepackage{mathrsfs}
\usepackage{graphicx}
\usepackage{amsmath}%
\usepackage{amssymb}%
\usepackage{graphicx}

\newtheorem{thm}{Theorem}[section]

\newtheorem{prop}[thm]{Proposition}
\newtheorem{lem}[thm]{Lemma}
\newtheorem{conj}[thm]{Conjecture}

\theoremstyle{definition}

\theoremstyle{remark}

\makeatletter
\let\c@equation\c@thm
\makeatother
\numberwithin{equation}{section}

\title{The Strong Chromatic Index of graphs with maximum degree $\Delta$}

\author{Chuanyun Zang}
\address
{Department of Mathematics and Statistics,\newline
\indent Georgia State University, Atlanta, GA 30303}
\thanks{Corresponding author: Chuanyun Zang, Georgia State University, Atlanta, GA, USA. Email: czang1@gsu.edu.}

\date{\today}

\begin{document}

\begin{abstract}

A strong edge-coloring of a graph $G$ is an edge-coloring such that no two edges of distance at most two receive the same color. The strong chromatic index $\chi'_s(G)$ is the minimum number of colors in a strong edge-coloring of $G$. P. Erd\H{o}s and J. Ne\v{s}et\v{r}il conjectured in 1985 that $\chi'_s(G)$ is bounded above by $\frac54\Delta^2$ when $\Delta$ is even and $\frac14(5\Delta^2-2\Delta+1)$ when $\Delta$ is odd, where $\Delta$ is the maximum degree of $G$. In this paper, we give an algorithm that uses at most $2\Delta^2-3\Delta+2$ colors for graphs with girth at least $5$. And in particular, we prove that any graph with maximum degree $\Delta=5$ has a strong edge-coloring with $37$ colors.

\end{abstract}

\maketitle

\section{Introduction}

A proper \emph{edge-coloring} is an assignment of colors to edges of a graph such that no two edges with common endpoint receive the same color. A \emph{strong edge-coloring} is an assignment of colors to edges of a graph such that no two edges of distance at most two receive the same color. Two edges are of distance at most two if and only if either they share an endpoint or one of their end points are adjacent. An induced matching is a set of edges such that no two edges are of distance at most two. Clearly, a strong edge coloring is an edge coloring in which every color class is an induced matching.
\emph{The strong edge chromatic number} of $G$, usually denoted by $\chi'_s(G)$, is the minimum number of colors in a strong edge-coloring of $G$. For example, the strong chromatic number of Petersen graph is $5$.
A partial strong edge-coloring is a strong edge-coloring except that some edges may be left uncolored.

Let $G$ be a graph with maximum degree $\Delta$. P. Erd\H{o}s and J. Ne\v{s}et\v{r}il \cite{FGST} proposed the following conjecture in $1985$.

\begin{conj}
\begin{equation*}
\chi'_s(G) \leq
\begin{cases}
 \frac54\Delta^2 &\text{if $\Delta$ is even }\\
 \frac14(5\Delta^2-2\Delta+1) &\text{ if $\Delta$ is odd.}\\

\end{cases}
\end{equation*}
\end{conj}

The conjectured bounds are best possible with the constructions obtained from a blowup of $C_5$. When $\Delta$ is even, expanding each vertex of a $5$-cycle into an independent set of cardinality $\Delta/2$ yields such a graph with $5\Delta^2/4$ edges. Similarly when $\Delta$ is odd, expanding each of two adjacent vertices into an independent set of cardinality $(\Delta+1)/2$ and each of the other three vertices of $C_5$ into independent set of cardinality $(\Delta-1)/2$ yields a graph with strong chromatic number $(5\Delta^2-2\Delta+1)/4$. Chung, Gy\v{a}rf\v{a}s, Trotter, and Tuza \cite{CGTT} proved that this operation gives the maximum number of edges in a $2K_2$-free graph with maximum degree $\Delta$.

The conjecture has been verified for graphs with maximum degree $\Delta \leq 3$. By using greedy edge coloring strategy, we can easily get $\chi'_s(G) \leq 2\Delta^2-2\Delta+1$. That implies the conjecture is true for $\Delta \leq 2$. For $\Delta=3$, it is proved by Andersen\cite{A} and independently, by Hor\v{a}k, He and Trotter \cite{HHT}, that $\chi'_s(G) \leq 10$ where $G$ is a graph with maximum degree $\Delta= 3$.

For $\Delta=4$, by the conjecture every graph must have a strong edge-coloring using $20$ colors. D. Cranston \cite{D} proved that any graph with maximum degree $4$ has a strong edge-coloring using at most $22$ colors. That is the best upper bound known for $\Delta=4$. For $\Delta=5$, by the conjecture every graph must have a strong edge-coloring using $29$ colors. The strong edge-coloring problem for $\Delta=4$ or $5$ is still open.

In this paper, inspired by papers of L. D. Anderson \cite{A} and D. Cranston \cite{C}, we want to get an upper bound for strong edge chromatic number of graphs with maximum degree $\Delta$. The \emph{girth} of a graph is the length of the shortest cycle. We prove that a $\Delta$-regular 
graph $G$ with girth at leat $5$  has a strong edge-coloring that uses $2\Delta^2-3\Delta+2$ colors. By applying this algorithm to graphs with maximum degree $5$, we obtain a strong edge-coloring using $37$ colors.

Our main results are as follows.
\begin{thm}\label{main1}
If $G$ is a graph with maximum degree $\Delta$ and girth at least $5$, then $G$ has a strong edge-coloring that uses $2\Delta^2-3\Delta+2$ colors.
\end{thm}

\begin{thm}\label{main2}
If $G$ is a graph with maximum degree $\Delta=5$, then $G$ has a strong edge-coloring that uses $37$ colors.
\end{thm}

\vspace{1pc}
\section{An observation for graphs with maximum degree $\Delta$}

Let $G$ be a graph with maximum degree $\Delta$. We refer to the color classes as the integers from $1$ to $2\Delta^2-3\Delta+2$. A greedy coloring strategy is to use the least color class that is not forbidden from use on an edge at the time the edge is colored, i.e., when coloring an edge $e=xy$, we need to forbid the colors that are already used by the edges incident to $x$ or $y$, as well as the colors by the edges having an end-vertex adjacent to $x$ or $y$. Define the neighborhood of $e$ as the set of edges that are incident to $e$, or has an end-vertex adjacent to some end-vertex of $e$, denoted by $N(e)$. Then $|N(e)| \leq 2\Delta(\Delta-1)$. Let $F(e)$ be the set of colors occurring on edges of $N(e)$; edges of $N(e)$ that are still uncolored do not contribute to $F(e)$, therefore $|F(e)| \leq |N(e)| \leq 2\Delta(\Delta-1)$.

Our aim is to find an order of the edges in which we can color the edges of $G$ one by one. Let $v$ be an arbitrary vertex of $G$. For $i=0, 1, 2,..$, let $D_i$ be the set of vertices of distance $i$ from $v$ and we call $D_i$ \emph{distance class i}. So $D_0=\{v\}$. For any edge $e$ of $G$, its distance denoted as $d_v(e)$ is the smallest distance among end-vertices of $e$. We say an edge order is \emph{compatible} with vertex $v$ if $e_1$ precedes $e_2$ in the order only when $d_v(e_1) \geq d_v(e_2)$. Intuitively, we color all the edges in distance class $i+1$ before we color any edge in distance class $i$.

\vspace{1pc}

\begin{lem}\label{lem1}

If $G$ is a graph with maximum degree $\Delta$, then $G$ has a strong edge-coloring that uses $2\Delta^2-3\Delta+1$ colors except that it leaves those edges incident to a single vertex.
\end{lem}

\begin{proof}
Let $v$ be a vertex of $G$. Greedily color the edges in an order that is compatible with $v$. If $e$ is an edge not incident to $v$, then $d_v(e)\geq 1$, and an end-vertex $x$ of $e$ with $x\in D_{d(e)}$ will be adjacent to a vertex $u$ in $D_{d(e)-1}$. When we color $e$, none of the $\Delta$ edges incident to $u$ has yet been colored, so at most $2\Delta^2-3\Delta$ edges of $N(e)$ have been colored, i.e $|F(e)|\leq 2\Delta^2-3\Delta$. Hence, we get a strong edge-coloring that uses $2\Delta^2-3\Delta+1$ colors except that it leaves those edges incident to $v$.

\end{proof}

\vspace{1pc}
\section{Graphs with maximum degree $\Delta$ and girth at least $5$}

\subsection{Graphs with $\delta < \Delta$}

\begin{lem}\label{lemnr}
Any graph with maximum degree $\Delta$ that has a vertex with degree at most $\Delta-1$ has a strong edge-coloring that uses $2\Delta^2-3\Delta+1$ colors.
\end{lem}

\begin{proof}
Let $v$ be the vertex with degree at most $\Delta-1$. Greedily color the edges in an order that is compatible with $v$, by Lemma \ref{lem1} we get a partial strong edge-coloring using $2\Delta^2-3\Delta+1$ except leaving those edges incident to $v$. Let $e_i$ be the edge incident to $v$, $|N(e_i)|\leq 2\Delta^2-3\Delta$, where $i=1,2,\dots,\Delta-1$. We can color those edges incident to $v$ in the order $e_1, e_2, e_3,...,e_{\Delta-1}$, and $|F(e_1)|\leq 2\Delta^2-3\Delta-\Delta+2$, $|F(e_2)|\leq 2\Delta^2-3\Delta-\Delta+3$, \dots, $|F(e_{\Delta-1})|\leq 2\Delta^2-3\Delta$, so there are colors available for each edge incident to $v$.

\end{proof}

\subsection{$\Delta$-regular graphs with girth at least $5$}

\begin{lem}\label{lemg5}

Any $\Delta$-regular graph with girth at least $5$ has a strong edge-coloring that uses $2\Delta^2-3\Delta+2$ colors.
\end{lem}

Before proving Lemma \ref{lemg5}, we first do some observations. Let $v$ be a vertex of $G$. We want to greedily color the edges in an order that is compatible with $v$. By Lemma \ref{lem1}, we get a partial strong edge-coloring that uses $2\Delta^2-3\Delta+1$ colors except that it leaves those edges incident to $v$. To finish the proof, we need to consider the local structure of those uncolored edges incident to $v$. By adding one more color class, we release $\Delta$ colors available to be greedily assigned to those edges incident to $v$.

Let $D_1, D_2$ be the vertex distance classes of $v$ with distance $1$ and $2$, respectively. Since the girth is at leat $5$, there are no induced edges within $D_1$, and any two distinct vertices in $D_1$ don't have common neighbor in $D_2$. Let $E[D_1, D_2]$ be the set of edges which have one end in $D_1$ and the other end in $D_2$.

\begin{prop}\label{prop}
By recoloring, we can assign the same color (say color $\alpha$) to $\Delta$ edges of $E[D_1, D_2]$.
\end{prop}

\begin{proof}
Let $D_1={w_1, w_2, ..., w_\Delta}$. For $i=1,2,3,\dots, \Delta$,  $w_i$ has $\Delta-1$ neighbors in $D_2$, and denote the set of these neighbors as $W_i$. Since there is no triangle in $G$, $w_i\cup W_i$ induces a $K_{1,\Delta-1}$. Now we give some observations as follows:

\begin{enumerate}
 \item [a.] $W_i\cap W_j=\emptyset$, for any $i\ne j$.
 \item [b.] no induced edges within $D_1$ .
 \item [c.] no induced edges within each $W_i, i=1,2,3,\dots,\Delta$ .
 \item [d.] $|N(u)\cap W_j| \leq 1$ for any $u \in W_i$ where $i\ne j$ and  $i, j=1,2,3,\dots,\Delta$.

\end{enumerate}

Let $N_2(u):=N(u)\cap D_2$ for any $u \in D_2$.
Our goal is to find an induced matching in $E[D_1, D_2]$ with size $\Delta$ and assign new color $\alpha$ to them. It is sufficient to find an independent set $V_0$ of size $\Delta$ consisting of exactly one vertex from each $W_i$ where $i=1,2,3,\dots, \Delta$.

\emph{Case 1. } $|N_2(u)| \leq 1$ for any $u$ in $D_2$. If $|N_2(u)|=0$ for any $u$ in $D_2$, i.e., there is no edge in $D_2$, then we can choose $\Delta$ edges in $E[D_1, D_2]$ by choosing one vertex from each $W_i, i=1,2,3,\dots, \Delta$. If there exists $u\in D_2$ such that $|N_2(u)|=1$, note that the set of edges in $D_2$ is an induced matching. We choose a vertex with one neighbor in $D_2$, say it is from $W_1$, denoted as $v_1$. Suppose the only neighbor of $v_1$ in $D_2$ is from $W_2$, then we can choose one vertex from $W_2$ which is not adjacent to $v_1$, denoted as $v_2$. This is possible since $W_1\cap W_2=\emptyset$. Consider the only neighbor of $v_2$ in $D_2$, if it is not in $W_1$, we may assume $N_2(v_2)\subset W_3$, then we choose one vertex from $W_3$ which is different from this neighbor, denoted as $v_3$; otherwise we can arbitrarily choose one vertex from $W_3$ with $\Delta$ choices. Continue this process, and each step we have at least $\Delta-1$ choices. So we get a vertex subset $V_0\subset D_2$ of size $\Delta$ such that $E[D_1, V_0]$ is an induced matching.

\emph{Case 2. } There exists a vertex $u$ in $D_2$ such that $|N_2(u)| \geq 2$. Let $v_1, v_2 \in N_2(u)$, and suppose $u \in W_\Delta, v_1 \in W_1, v_2 \in W_2$. It is obvious that$v_1, v_2$ are nonadjacent otherwise there is a triangle. Let $V_1=\{v_1\}$, we will choose vertices sequentially as follows:

If we already have $V_{k-1}=\{v_1, v_2, \dots, v_{k-1} \}$, then choose $v_k \in W_k\setminus N_2(V_{k-1})$ and let $V_k=V_{k-1} \cup \{v_k \}$.
This process is possible since $|W_k| = \Delta -1$ and $|W_k \cap N_2(V_{k-1})| \leq k-1$ because of observation (d), we get $|W_k\cap N_2 (V_{k-1})|\ge (\Delta-1)-(k-1) \geq 1$ when $k \leq \Delta-1$. When $k=\Delta$, since $N_2(v_1, v_2)\cap W_\Delta=\{u\}$, we have $|W_\Delta \setminus N_2(V_{\Delta-1})| \leq (\Delta-1)-1=\Delta-2 < |W_{\Delta}|$. So we choose $v_\Delta\in W_{\Delta}$ and let $V_0=V_{k-1} \cup {v_\Delta}$, and $E[D_1, V_0]$ is an induced matching.
\end{proof}

\begin{proof}[Proof of lemma of Lemma \ref{lemg5}]

First by Lemma \ref{lem1}, we get a partial strong edge-coloring $\pi$ with $2\Delta^2-3\Delta+1$ colors except that it leaves those edges incident to some vertex $v$. Now consider the local structure within 2 distance classes from $v$, by Proposition \ref{prop}, we can assign a new color $\alpha$ to $\Delta$ edges in $E[D_1, D_2]$ and release those colors used by these $\Delta$ edges in $\pi$. By greedily assign these released color to those $\Delta$ edges incident to $v$, we obtain a strong edge-coloring that uses $2\Delta^2-3\Delta+2$ colors.
\end{proof}

\vspace{1pc}

\section{Graphs with maximum degree $\Delta=5$}

Lemma \ref{lem1} with $\Delta=5$ provides a partial strong edge-coloring with $36$ colors. So we only need to consider the local structure within distance $2$ from a single vertex $v$. When the girth of the graph is at least $5$, Theorem \ref{main2} can be obtained from Lemma \ref{lemg5} since $2\Delta^2-3\Delta+2=37$ with $\Delta=5$. When there exists a vertex with degree less than $5$, Theorem \ref{main2} is true by Lemma \ref{lemnr}. Therefore the remaining cases are the 5-regular graphs with the girth at most $4$.

\subsection{Graphs with $\Delta=5$ and girth $3$} We have the following lemma.

\begin{lem}
If $G$ is a $5$-regular graph with girth $3$, then $G$ has a strong edge-coloring that uses $37$ colors.
\end{lem}

\begin{proof}
Start from a triangle $\{v_1, v_2, v_3\}$ with edges $c_1=v_1v_2, c_2=v_2v_3, c_3=v_3v_1$, First by Lemma \ref{lem1}, we get a partial strong edge-coloring using $36$ colors with the edges incident to $v_1$ uncolored. Now release the colors used by the edges incident to $v_2$ and edges incident to $v_3$, and we have $12$ uncolored edges. Assign colors to all the edges incident to the triangle first and then the edges on the triangle. Since any edge $e$ incident to the triangle, we have $|N(e)| = 39$, and $|F(e)|\leq 39-3 < 37$, we can greedily color it. For $i=1,2,3, |N(c_i)|= 35, |F(c_i)| \leq 35 <37$, we can also greedily color $c_1, c_2, c_3$.
\end{proof}

\subsection{Graphs with $\Delta=5$ and girth $4$}

\begin{lem}
If $G$ is a $5$-regular graph with girth $4$, then $G$ has a strong edge-coloring that uses $37$ colors.
\end{lem}
\begin{proof}
Let $G$ be a $5$-regular graph with girth $4$. Let $v$ be a vertex on a $4$-cycle of $G$. Color the edges in an order compatible with $v$, by Lemma \ref{lem1}, we get a partial strong edge-coloring with $36$ colors. 

Let $e_i=vw_i$ and $W_i=N_(w_i)\cap D_2$ where $i=1,2,3,4,5$. Suppose $w_1$ and $w_2$ have a common neighbor in $D_2$. Because girth is $4$, we have $D_1$ is independent. 
Observation:

(a) Since $|N(e_i)|=40-|E(D_1\setminus\{w_i\},W_i)|$, if $|E(D_1\setminus\{w_i\},W_i)|\ge 4$ then we can greedily color $e_i$ where $i=1,\dots,5$. 

(b) $|D_2|\ge 7$. Otherwise, we can greedily color $e_i$ where $i=1,\dots,5$. 

Since $|F(e_i)|\le 39-4=35$ for $i=1,2$ and $|F(e_i)|\le 40-4=36$ when $i=3,4,5$, we'll have a similar argument with the proof of Lemma \ref{lemg5} to show that we can reassign a new color to $3$ edges in $E(D_1, D_2)$, otherwise the neighborhood of $e_i$ where $i=1,2,3,4,5$ is small enough for us to greedily color it. If we have at least three $W_i$ that contains vertices with no neighbor in $D_2\cup W_i\setminus\{w_i\}$, then we can choose an induced matching in $E(D_1,D_2)$ of size at least three.
Otherwise we have the following cases.

\emph{Case 1. } There are two $W_i$, say $W_4, W_5$ that contains vertices with no neighbor in $D_2\cup W_i\setminus\{w_i\}$. Choose $v_4\in W_4, v_5\in W_5$ such that $N(v_i)\cap (D_2\cup W_i\setminus\{w_i\})=\emptyset$ for $i=4,5$. So we only need to choose one vertex $v_i\in W_i$ for some $i=1,2,3 $ such that $w_4v_4, w_5v_5, w_iv_i$ form an induced matching. If such vertex does not exist, then any $v\in \cup_{i=1}^3 W_i$, $v$ is adjacent to $w_1$ or $w_2$. Because $w_1$ and $w_2$ have a common neighbor in $D_2$, we have $|\cup_{i=3}^5 W_i|\le 6$. Therefore, $|E(D_1\setminus\{w_i\},W_i)|\ge 4$ for each $i=1,2,3$.

\emph{Case 2.} There one $W_i$, say $W_5$ that contains vertices with no neighbor in $D_2\cup W_i\setminus\{w_i\}$, and choose one such vertex as $v_5$. By observation (b), we have at least one vertex $v\in D_2\subset W_5$, say $v\in W_4$, then $v_5w_5$ and $vw_4$ is an induced matching. Following the same argument in Case 1, we either find an induced matching of size 3 or $|E(D_1\setminus\{w_i\},W_i)|\ge 4$ for each $i\in [5]$.

\emph{Case 3.} There is no $W_i$ that contains vertices with no neighbor in $D_2\cup W_i\setminus\{w_i\}$. Let $v$ be the common neighbor of $w_1$ and $w_2$ in $D_2$. If $v$ is adjacent to all $w_i$, then $|E(D_1\setminus\{w_i\},W_i)|\ge 4$ for each $i=1,\dots, 5$. We may assume $v$ is not adjacent to $w_5$. Since $w_i$ has at most $3$ neighbors in $W_5$ for $i=1,2$, we have either $v$ is adjacent to at least one vertex in $W_5$ or at least one of $w_1, w_2$ is adjacent to any vertex in $W_5$. In either case, we can find an induced matching of size $3$ in $E(D_1,D_2)$ otherwise $|E(D_1\setminus\{w_i\},W_i)|\ge 4$ for each $i\in [5]$. 
\end{proof}

\vspace{1pc}

\section{Acknowledgments} 
\vspace{1pc}
This paper was conducted as a project while the author was taking a class of Dr. Guantao Chen. The present author would like to thank Ilkyoo Choi and Suil Oh for the useful discussion.

\end{document}